\documentclass[11pt]{amsart}
\usepackage[utf8]{inputenc}
\usepackage{amssymb, enumerate, graphicx}

\newtheorem{theorem}{Theorem}[section] 
\newtheorem{lemma}[theorem]{Lemma}

\theoremstyle{definition}

\newtheorem{example}[theorem]{Example}

\theoremstyle{remark}
\newtheorem{remark}[theorem]{Remark}
\numberwithin{equation}{section}

\DeclareMathOperator{\re}{Re}
\DeclareMathOperator{\im}{Im}
\DeclareMathOperator{\res}{res}

\title{Algebraic structure of the range of a trigonometric polynomial}

\author{Leonid V. Kovalev}
\address{215 Carnegie, Mathematics Department, Syracuse University, Syracuse, NY 13244, USA}
\email{lvkovale@syr.edu}
\thanks{L.V.K. supported by the National Science Foundation grant DMS-1764266.}

\author{Xuerui Yang}
\address{215 Carnegie, Mathematics Department, Syracuse University, Syracuse, NY 13244, USA}
\email{xyang20@syr.edu}
\thanks{X.Y. supported by Young Research Fellow award from Syracuse University.}



\begin{document}
\baselineskip6mm

\maketitle

\begin{abstract}
    The range of a trigonometric polynomial with complex coefficients can be interpreted as the image of the unit circle under a Laurent polynomial. We show that this 
    range is contained in a real algebraic subset of the complex plane. Although the containment may be proper, the difference between the two sets is finite, except for polynomials with certain symmetry. 
\end{abstract}

\section{Introduction}

In 1976 Quine~\cite[Theorem 1]{Quine76} proved that the image of the unit circle $\mathbb T$ under an algebraic polynomial $p$ of degree $n$ is contained in a real algebraic set $V = \{(x, y)\in \mathbb R^2 \colon q(x, y) = 0\}$ where $q$ is a polynomial of degree $2n$. In general $p(\mathbb T)$ is a proper subset of $V$, but we will show that $V\setminus p(\mathbb T)$ is finite, and that $V= p(\mathbb T)$ whenever $V$ is connected. 

Consider a trigonometric polynomial $P(t) = \sum_{k=-m}^n a_k e^{ikt}$, $t\in \mathbb R$, with complex coefficients $a_k$. It is natural to require $a_{-m}a_n\ne 0$ here. The range of $P$ is precisely the image of the unit circle $\mathbb T$ under the Laurent polynomial $p(z)= \sum_{k=-m}^n a_k z^k$. This motivates our investigation of $p(\mathbb T)$ for Laurent polynomials. Our main result, Theorem~\ref{algebraic-thm}, asserts that $p(\mathbb T)$ is contained in the zero set $V$ of a polynomial of degree $2\max(m, n)$. This matches Quine's theorem in the case of $p$ being an algebraic polynomial, i.e., $m=0$. The difference $V\setminus p(\mathbb T)$ is finite when $m\ne n$, but may be infinite when $m=n$. 

In Section~\ref{sec:exceptional} we investigate the exceptional case when $V\setminus p(\mathbb T)$ is infinite, and relate it to the properties of the zero set of a certain harmonic rational function. The structure of zero sets of such functions is a topic of current interest with applications to gravitational lensing~\cite{BHJR, KhavinsonNeumann}.

Finally, in Section~\ref{sec:intersect} we use the algebraic nature of the polynomial images of $\mathbb T$ to estimate the number of intersections of two such images, i.e., the number of shared values of two trigonometric polynomials.

\section{Algebraic nature of polynomial images of circles}

By definition, a real algebraic subset of $\mathbb R^2$ is a set of the form $\{(x, y) \in \mathbb R^2 \colon q(x, y) = 0\}$ where $q\in \mathbb R[x, y]$ is a polynomial in $x, y$. 
Consider a Laurent polynomial
\begin{equation}\label{Lpoly}
p(z) = \sum_{k=-m}^n a_k z^k,\quad z\in \mathbb C\setminus \{0\}, 
\end{equation}
where $m\ge 0$, $n\ge 1$, and $a_{-m} a_n\ne 0$. This includes the case of algebraic polynomials ($m=0$), because the condition $a_0\ne 0$ can be ensured by adding a constant to $p$, which does not affect the algebraic nature of $p(\mathbb T)$. Since we are interested in the image of the unit circle, which is invariant under the substitution of $z^{-1}$ for $z$, it suffices to consider the case $m\le n$.  

\begin{theorem}\label{algebraic-thm} Let $p$ be the Laurent polynomial~\eqref{Lpoly} with $m\le n$. 
\begin{enumerate}[(a)]
    \item The image of $\mathbb T$ under $p$, is contained in the zero set $V$ of some polynomial $h\in \mathbb R[x, y]$ of degree $2n$.
    \item If $h$ is expressed as a polynomial  
$h_{\mathbb C} \in  \mathbb C[w, \overline{w}]$ via the substitution $w=x+iy$, the degree of $h_{\mathbb C}$ in each of the variables $w$ and $\overline{w}$ separately is $m+n$.
\item If $m<n$, then the set $V\setminus p(\mathbb T)$ is finite.
\item In the case $m=n$ the set $V\setminus p(\mathbb T)$ is finite if and only if $V$ is bounded.
\end{enumerate} 
\end{theorem}

The proof of Theorem~\ref{algebraic-thm} involves two polynomials 
\begin{equation}\label{polys-g-gs}
g(z)=z^{m}(p(z)-w) \quad \text{and} \quad  g^*(z)= z^{n+m}\overline{g(1/\bar z)} = 
z^n\overline{(p(1/\bar{z})-w)}
\end{equation}
which are the subject of the following lemma. 

\begin{lemma}\label{degrees} The resultant $h_{\mathbb C} = \res(g, g^*)$ of the polynomials~\eqref{polys-g-gs} is a polynomial in $\mathbb C[w, \overline{w}]$ of  degree $2n$.  Moreover, $h_{\mathbb C}$ has degree $m+n$ in each of the variables $w$ and $\overline{w}$ separately. Finally, $h(x, y) := h_{\mathbb C}(x+iy, x-iy)$ is a polynomial of degree $2n$ in $\mathbb R[x, y]$. 
\end{lemma}

\begin{proof} 
Both $g$ and $g^*$ are polynomials of degree $m+n$ in $z$, except for the case $m=0$ and $w=a_0$ which we ignore in this proof because considering a generic $w$ is enough. By definition, the resultant of $g$ and $g^*$ is the determinant of the following matrix of size $2(m+n)$. 
\begin{equation}\label{sylvester-matrix}
R = 
\begin{pmatrix}
a_{-m} & \cdots & \cdots & a_0 - w &  \cdots  & a_{n} & 0 & 0 \\
0& \ddots & & & \ddots & & \ddots & 0 \\
0 & 0 & a_{-m} & \cdots  & \cdots  & a_0 - w  & \cdots  & a_{n} \\ 
\overline{a_{n}} & \cdots  &\overline{a_0} - \overline{w} &  \cdots  &  \cdots  & \overline{a_{-m}} & 0 & 0 \\
0& \ddots & & \ddots & & & \ddots & 0 \\
0 & 0 & \overline{a_{n}} & \cdots  & \overline{a_0} - \overline{w}  &  \cdots  & \cdots  & \overline{a_{-m}}  \\
\end{pmatrix}
\end{equation}

All appearances of $w$ or $\overline{w}$ in $R$ are in the columns numbered $m+1$  through $m+2n$, which are the middle $2n$ columns of matrix $R$. Therefore, $h_{\mathbb C}$ is a polynomial of degree at most $2n$. 

Let us first prove that $h_{\mathbb C}$ has degree $n+m$ in each variable separately. It obviously cannot be greater than $n+m$, since each of $w$ and $\overline{w}$ appears $n+m$ times in the matrix. The position of $a_0-w$ in the top half of the matrix shows that the Leibniz formula for $\det R$ contains the term $\pm \overline{a_n}^m\overline{a_{-m}}^n (a_0-w)^{n+m}$ and no other terms with the monomial $w^{n+m}$. Therefore, the coefficient of $w^{n+m}$ in $h$ is $\pm \overline{a_n}^m\overline{a_{-m}}^n \ne 0$. Similarly, 
 the coefficient of $\overline{w}^{2n}$ in $h$ is $\pm a_{-m}^n a_n^m  \ne 0$. This proves that $h_{\mathbb C}$ has degree $n+m$ in $w$ and $\overline{w}$ separately.

When $m=n$, the preceding paragraph shows that $h_{\mathbb C}$ has degree $2n$ in $w$ and $\overline{w}$ separately, which implies $\deg h = 2n$. 

We proceed to prove $\deg h_{\mathbb C} = 2n$ in the case $m<n$. Let $R_1$ be the matrix obtained from $R$ by replacing all constant entries in the columns $m+1, \dots, m+2n$ by $0$. Since the cofactor of any of the entries we replaced is a polynomial of degree less than $2n$, the difference $\det R - \det R_1$ has degree less than $2n$. Thus, it suffices to show that $\det R_1 $ has degree $2n$. 
When deriving a formula for $\det R_1 $ we may assume $w\ne a_0$. Let us focus on the columns of $R_1$ numbered $m+1, \dots,2m$: the only nonzero entries at these columns are: 
\begin{itemize}
    \item $a_0 - w $ at 
$(j-m, j)$ for $m+1\le j\le 2m$; 
    \item $\overline{a_0} -\overline{ w}$ at 
$(j+m, j)$ for $n+1\le j\le 2m$.
\end{itemize}

We can use column operations to eliminate all nonzero entries in the upper-left $m\times m$ submatrix of $R_1$. Since this submatrix is upper-triangular, the process only involves adding some multiples of $j$th column with $m+1\le j\le 2m$ to columns numbered $k$ where $j-m\le k \le m$.  Such a column operation also affects the bottom half of the matrix, where we add a multiple of 
the entry $(j+m, j)$ to the entry $(j+m, k)$. 
Since $(j+m)-k \le j+m - (j-m) = 2m < n+m$, the affected entries of the bottom half are strictly above the diagonal $\{(n+m+j, j)\colon 1\le j\le m\}$ which is filled with the value $a_n$. In conclusion, these column operations do not substantially affect the upper-triangular submatrix 
formed by the entries $(i, j)$ with $n+m+1\le i\le n+2m$, $1\le j\le m$, in the sense that the submatrix remains upper-triangular and its diagonal entries remain equal to $a_n$. 

Similar column operations on the right side of the matrix eliminate all nonzero entries in the bottom right $m\times m$ submatrix of $R_1$. Let $R_2$ be the resulting matrix: 
\[
R_2 = 
\begin{pmatrix}
0 & \cdots  & a_0 - w &  \cdots  &  \cdots  & 0 & 0 & 0 \\
0& \ddots & & \ddots & & & \ddots & 0 \\
0 & 0 & 0 & \cdots  & a_0 - w  &  \cdots  & \cdots  & a_n  \\
\overline{a_n} & \cdots & \cdots & \overline{a_0} - \overline{w} &  \cdots  & 0 & 0 & 0 \\
0& \ddots & & & \ddots & & \ddots & 0 \\
0 & 0 & 0 & \cdots  & \cdots  & \overline{a_0} - \overline{w}  & \cdots  & 0 \\ 
\end{pmatrix}
\]
We claim that $\det R_2 = \pm |a_n|^{2m} |a_0-w|^{2n}$. Indeed, the first $m$ columns of $R_2$ contain only an upper-triangular submatrix with $\overline{a_n}$ on the diagonal; the last $m$ columns contain only a lower-triangular matrix with $a_n$ on the diagonal. After these are accounted for, we are left with a $2n\times 2n$ submatrix in which every row has exactly one nonzero element, either $a_0 - w$ or its conjugate. This completes the proof of $\deg h_{\mathbb C} = 2n$. 

Define $h(x,y) = h_{\mathbb C}(x+iy, x-iy)$ for real $x, y$. We claim that $h$ is real-valued, and thus has real coefficients. Recall (e.g.,~\cite[p.~11]{Orzech}) that the resultant can be expressed in terms of the roots of the polynomials $g, g^*$. 
Let $z_1, \dots, z_{n+m}$ be the roots of $g$ listed with multiplicity.
To simplify notation, we separate the cases $m>0$ and $m=0$. 

\textbf{Case $m>0$}. We have $\prod_{k=1}^{m+n} z_k = (-1)^{n+m}a_{-m}/a_n$; in particular, $z_k\ne 0$ for all $k$. It follows from~\eqref{polys-g-gs} that $g^*$ has roots $1/\overline{z_k}$ for $k=1, \dots, n+m$. The leading terms of $g$ and $g^*$ are $a_n$ and $\overline{a_{-m}}$, respectively. Thus, 
\begin{equation}\label{res-zeros}
    \begin{split}
\res(g, g^*) & = (\overline{a_{-m}}a_n )^{m+n} \prod_{i, j=1}^{n+m} (z_i - 1/\overline{z_j}) = (\overline{a_{-m}}a_n )^{n+m} \prod_{i, j=1}^{n+m} \frac{z_i \overline{z_j} - 1}{\overline{z_j}} \\ 
& = (\overline{a_{-m}}a_n )^{n+m} \left(\prod_{j=1}^{m+n} \overline{z_j}\right)^{-(n+m)}
\prod_{i, j=1}^{n+m} (z_i \overline{z_j} - 1) 
\\ & = (-1)^{n+m} (\overline{a_{-m}}a_n )^{m+n} (\overline{a_n/a_{-m}})^{n+m} 
\prod_{i, j=1}^{n+m} (z_i \overline{z_j} - 1) 
\\ & = (-1)^{n+m} |a_n|^{2(m+n)} \prod_{i, j=1}^{n+m} (z_i \overline{z_j} - 1) 
\end{split}
\end{equation}
The latter product is evidently real. 

\textbf{Case $m=0$}. We have $\prod_{k=1}^{m+n} z_k = (-1)^{n}(a_{0}-w)/a_n$; in particular, $z_k\ne 0$ for all $k$ provided that $w\ne a_0$. The rest of the proof goes as in case $m>0$, with $a_{-m}$ replaced by $a_0-w$ throughout. Since $a_{-m}$ cancels out at the end of~\eqref{res-zeros}, the conclusion that $h$ is real-valued still holds.
\end{proof}

The following description of the local structure of the zero set of a complex-valued harmonic function is due to Sheil-Small (unpublished) and appears in~\cite{Wilmshurst}. 

\begin{theorem}\label{Wilmshurst}\cite[Theorem 3]{Wilmshurst} 
Let $\Omega\subset \mathbb C$ be a domain and let $f\colon \Omega\to\mathbb C$ be a harmonic function. Suppose that the points $\{z_k\}_{k=1}^\infty$ are distinct zeroes of $f$ which converge to a point $z^*\in \Omega$. Then $z^*$ is an interior point of a simple analytic arc $\gamma$ which is contained in $f^{-1}(0)$ and contains infinitely many of the points $z_k$. 
\end{theorem}

The fact that $z_k\in \gamma$ for infinitely many $k$ is not stated in \cite[Theorem 3]{Wilmshurst} but is a consequence of the proof. 

\begin{proof}[Proof of Theorem~\ref{algebraic-thm}]
(a)-(b) Suppose $w\in p(\mathbb T)$. Then the rational functions $p(z)-w$ and $\overline{p(1/\bar z)-w}$ have a common zero, namely, any preimage of $w$ that lies on $\mathbb T$. Consequently, the polynomials~\eqref{polys-g-gs} have a common zero, which implies that their resultant $h_{\mathbb C}=\res(g, g^*)$ vanishes at $w$. The claims (a) and (b) follow from Lemma~\ref{degrees}. For future references, note that the zero set of $h$ can be written as
\begin{equation}\label{V-structure}
V = h^{-1}(0) = p(E),\quad \text{where } 
E = \{z\in \mathbb C\setminus \{0\} \colon p(z) = p(1/\bar z)\}. 
\end{equation}

(c) In view of~\eqref{V-structure}, to prove that $V\setminus p(\mathbb T)$ is finite it suffices to show that $E\setminus \mathbb T$ is finite.  Let $q(z) = p(z) - p(1/\bar z)$ which is a harmonic Laurent polynomial. Since $m<n$, it follows that $q(z) = p(z) + O(|z|^m) = a_n z^n + O(|z|^{n-1})$ as $|z|\to \infty$. Thus $E$ is a bounded set. By symmetry, $E$ is also bounded away from $0$. 

Suppose that $E\setminus \mathbb T$ is infinite. Then it contains a convergent sequence of distinct points $z_k\to z^* \ne 0$. By Theorem~\ref{Wilmshurst} there exists a simple analytic arc $\Gamma$ such that $g_{|\Gamma} = 0$ and $z^*$ is an interior point of $\Gamma$. In the case $z^*\in \mathbb T$, the arc $\Gamma$ is not a subarc of $\mathbb T$, because it contains infinitely many of the points $z_k$ which are not on $\mathbb T$. By virtue of its analyticity, $\gamma$ has finite intersection with $\mathbb T$. By shrinking $\gamma$ we can achieve that $\gamma \cap \mathbb T = \{z^*\}$ if  $z^*\in \mathbb T$, and $\gamma \cap \mathbb T = \emptyset$ otherwise. 

Since the endpoints of $\gamma$ lie in $E\setminus \mathbb T$, the process described above can be iterated to extend $\gamma$ further in both directions. This continuation process can be repeated indefinitely. Since $E$ is bounded, we conclude that $E$ contains a simple closed analytic curve $\Gamma$, as in the proof of ~\cite[Theorem~4]{Wilmshurst}. 

If $\Gamma$ does not surround $0$, then the maximum principle yields $q\equiv 0$ in the domain  enclosed by $\Gamma$, which is impossible since $q$ is nonconstant. If $\Gamma $ surrounds $0$, then the complement of $\Gamma\cup \mathbb T$ has a connected component $G$ such that $0\notin G$. The maximum principle yields  $q\equiv 0$ in $G$, a contradiction. The proof of (b) is complete. 

(d) The proof of (c) used the assumption $m<n$ only to establish that the set $E$ in~\eqref{V-structure} is bounded. Thus, the conclusion still holds if $m=n$ and $E$ is a bounded set.  Recalling that $V = p(E)$ and $|p(z)|\to\infty$ as $|z|\to\infty$, we find that $E$ is bounded whenever $V$ is bounded.  

Finally, if $V$ is an unbounded set, then $V\setminus p(\mathbb T)$ must be infinite because $p(\mathbb T)$ is bounded. \end{proof}

Since a real algebraic set has finitely many connected components~\cite[Theorem~3]{Whitney}, it follows from Theorem~\ref{algebraic-thm} that when $V\setminus p(\mathbb T)$ is finite, the set $p(\mathbb T)$ coincides with one of the connected components of $V$, and the other components of $V$ are singletons. The number of singleton components of $V$ can be arbitrarily large, even when $p$ is an algebraic polynomial.

\begin{remark} For every integer $N$ there exists a polynomial $p$ such that the set $V\setminus p(\mathbb T)$ described in Theorem~\ref{algebraic-thm} contains at least $N$ points.
\end{remark}

\begin{proof} Let $a_1, \dots, a_N$ be distinct complex numbers with $0<|a_k|<1$ for $k=1, \dots, N$. Using Lagrange interpolation, we get a polynomial $q$ of degree $2N-1$ such that $q(a_k)=q(1/\bar a_k) = k$ for $k=1, \dots, N$. Let $r$ be a polynomial of degree $2N$ with zeros at the points $a_k$ and $1/\bar a_k$, $k=1,\dots, N$. Since $\inf_{\mathbb T}|r|>0$, for sufficiently large constant $M$ the polynomial $p = q+Mr$ satisfies 
$q(a_k)=q(1/\bar a_k) = k$ for $k=1, \dots, N$, as well as $|p(z)|>N$ for $z\in \mathbb T$. It follows that the algebraic set $V$, as described by~\eqref{V-structure}, contains the points $1, \dots, N$, none of which lie on the curve $p(\mathbb T)$.
\end{proof}

\section{Examples}\label{sec:example} 

First we observe that $p(\mathbb T)$ need not be a real algebraic set, even for a quadratic polynomial $p$. 

\begin{example} Let $p(z)=z^2+3z+1$. Then $p(\mathbb T)$ is not a real algebraic set. 
\end{example}

\begin{proof} Direct computation of the polynomial $h$ in Theorem~\ref{algebraic-thm} yields
\begin{equation}\label{h-example}
\begin{split} h(x, y) & =
\det \begin{pmatrix}
1 - w & 3 & 1 & 0 \\
0 & 1 - w & 3 & 1 \\
1 & 3 & 1- \overline{w}  & 0 \\
0 & 1  & 3 & 1- \overline{w} \\
\end{pmatrix}
\\ & = 
x^4 + 2 x^2 y^2  + y^4
- 4 x^3 - 4 x y^2 - 5 x^2 - 9 y^2   
\end{split}
\end{equation}
where $w = x+iy$. 
By Theorem~\ref{algebraic-thm} the set $h^{-1}(0)$ contains $p(\mathbb T)$. Since $p\ne 0$ on $\mathbb T$,   we have $0\in h^{-1}(0) \setminus p(\mathbb T)$. If $p(\mathbb T)$ was an algebraic set, then $V$ would be reducible. However, $h$ is an irreducible polynomial. Indeed, the fact that the zero set of $h$ is bounded implies that any nontrivial factorization $h=fg$ would have $\deg f = \deg g = 2$. This means that $V$ is the union of two conic sections, which it evidently is not, as $p(\mathbb T)$ is not an ellipse.
\end{proof}

\begin{figure}[h]
    \centering
    \includegraphics[width=0.4\textwidth]{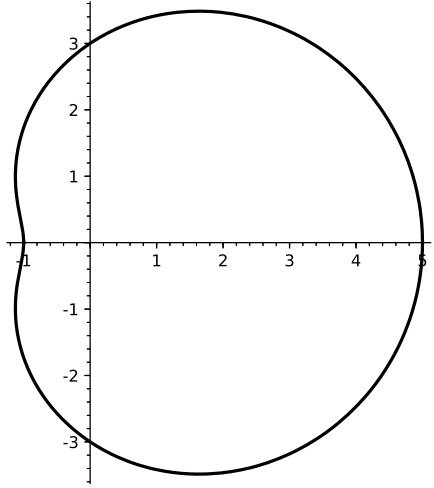}
    \caption{Non-algebraic image of the circle}
    \label{fig:example}
\end{figure}

According to Theorem~\ref{algebraic-thm}, the set $p(\mathbb T)$ can be completed to a real algebraic set by adding finitely many points, provided that $p$ is either an algebraic polynomial or a Laurent polynomial with $m<n$. The following example shows that the case $m=n$ is indeed exceptional. 

\begin{example}\label{line-segment}
Let $p(z)=z + z^{-1}$. Then $p(\mathbb T)$ is the line segment $[-2, 2]$. The smallest real algebraic set containing $p(\mathbb T)$ is the real line $\mathbb R$. 
\end{example}

The claimed properties of Example~\ref{line-segment} are straightforward to verify. In addition, the polynomial $h$ from Theorem~\ref{algebraic-thm} can be computed as $h(x, y) = -4y^2$ which shows that $h$ is not necessarily irreducible. 

\section{Zero set of harmonic Laurent polynomials}\label{sec:exceptional}

The relation~\eqref{V-structure} highlights the importance of the zero set of the harmonic Laurent polynomial $P(z)=p(z)-p(1/\bar{z})$ where $p$ is a Laurent polynomial. It is not a trivial task to determine whether a given harmonic Laurent polynomial has unbounded zero set: e.g., Khavinson and Neumann~\cite{KhavinsonNeumann} remarked on the varied nature of zero sets for rational harmonic functions in general.   In this section we develop a necessary condition, in terms of the coefficients of $p$, for the function $P$ to have an unbounded zero set. 

Suppose that $p$ is a Laurent polynomial~\eqref{Lpoly} such that the associated function $P(z)=p(z)-p(1/\bar{z})$ has unbounded zero set. Consider the algebraic part of $P$, namely 
\begin{equation}\label{algebraic-part}
q(z)=\sum_{k=1}^n a_k z^k - \sum_{k=1}^m a_{-k}{\bar{z}}^k.
\end{equation}
Then $q$ is a harmonic polynomial such that $\liminf_{z\to \infty}|q(z)|$ is finite. In other words, $q$ is not a proper map of the complex plane. 

One necessary condition is immediate: if  $m<n$, then $|q(z)| = a_n |z|^n + o(|z|^n)$ as $z\to\infty$. Thus, $P$ can only have unbounded zero set if $m=n$.  

We look for further conditions on a harmonic polynomial that ensure that it is a proper map of $\mathbb R^2$ to $\mathbb R^2$. More generally, given a polynomial map $F=(F_1, \dots, F_n)\colon \mathbb R^n\to \mathbb R^n$, let us decompose each component $F_k$ into homogeneous polynomials, and let $\mathcal H(F_k)$ be the homogeneous term of highest degree in $F_k$.  
Write $\mathcal H(F)$ for $(\mathcal H(F_1), \dots, \mathcal H({F_n}))$, so that $\mathcal H(F)$ is also a polynomial map of $\mathbb R^n$. The following result is from~\cite{Essen}, Lemma 10.1.9.  

\begin{lemma}\label{Campbell-general} [L. Andrew Campbell] If $\mathcal H({F})$ does not vanish in $\mathbb R^n\setminus \{0\}$, then $F\colon \mathbb R^n\to\mathbb R^n$ is a proper map, that is $|F(x)|\to \infty$ as $|x|\to\infty$. 
\end{lemma}

Lemma~\ref{Campbell-general} can be restated in a form adapted to harmonic polynomials in $\mathbb C$.

\begin{lemma}\label{Campbell} Consider a harmonic polynomial $q(z) = \sum_{k=0}^n (a_k z^k + b_k \bar z^k)$ of degree $n\ge 1$ as a map from $\mathbb C$ to $\mathbb C$. 
\begin{enumerate}[(a)]
    \item If $|a_n|\ne |b_n|$, then $q$ is proper.
    \item  If $|a_n|=|b_n|$, let $\eta\in \mathbb T$ be such that $\eta a_n = \overline{\eta b_n}$. If $\eta a_k = \overline{\eta b_k}$ for $k=1, \dots, n$, then $q$ is not proper. Otherwise, let $K$ be the largest value of $k$ such that  
$\eta a_k \ne  \overline{\eta b_k}$. If there is no $z\ne 0$ such that 
\[
\re(\eta a_n z^n) = 0 = \im((\eta a_K -  \overline{\eta b_K}) z^K) 
\]
then $q$ is proper. 
\end{enumerate} 
\end{lemma}

\begin{proof} Part (a) follows from the reverse triangle inequality: $|q(z)| \ge \big||a_n| - |b_n|\big| |z|^n + o(|z|^n)$ as $n\to\infty$. 
 To prove part (b), observe that
 \begin{equation}\label{Campbell-proof}
 \im(\eta q(z)) = \sum_{k=0}^n \im((\eta a_k -  \overline{\eta b_k} z^k)
  \end{equation}
 If $\eta a_k = \overline{\eta b_k}$ for $k=1, \dots, n$, then $\im(\eta q)$ is constant, which means that up to a constant term, $\eta q$ is a real-valued harmonic function. By Harnack's inequality, a nonconstant harmonic function $h\colon \mathbb C\to\mathbb R$ must be unbounded from above and from below, and therefore $q^{-1}(0)$ is an unbounded set. Since $q$ is constant on an unbounded set, it is not a proper map. 
 
 Finally, suppose that $K$, as defined in (b), exists. It follows from~\eqref{Campbell-proof} that 
 \[
 \mathcal H( \im(\eta q(z)) ) = \im((\eta a_K -  \overline{\eta b_K}) z^K) 
 \]
 Since also
 \[
\mathcal H( \re(\eta p(z)) ) = \re((\eta a_n +  \overline{\eta b_n}) z^n) = 2\re(\eta a_n z^n)
 \]
the last statement in (b) follows by applying Lemma~\ref{Campbell-general}  to $(\re(\eta q), \im(\eta q))$ considered as a map of $\mathbb R^2$ to  $\mathbb R^2$.
\end{proof}
 
We are now ready to apply  Lemma~\ref{Campbell} to the special case $P(z) = p(z)-p(1/\bar z)$ where $p$ is a Laurent polynomial. Recall that in view of Theorem~\ref{algebraic-thm} and the relation~\eqref{V-structure} the following result describes when the image $p(\mathbb T)$ has infinite complement in the real algebraic set $V$ containing it.  

\begin{theorem}\label{unbounded-thm} Given a Laurent polynomial $p(z) = \sum_{k=-n}^n a_k z^n$ with $a_n a_{-n} \ne 0$, let $P(z) = p(z) - p(1/\bar z)$. If the zero set of $P$ is unbounded, then one of the following holds:  
\begin{enumerate}[(a)]
\item $p(\mathbb T)$ is contained in a line; 
\item There exists $\eta\in \mathbb T$ such that $\eta a_n + \overline{\eta a_{-n}} = 0$. Furthermore, there is an integer $k\in \{1, \dots, n-1\}$ such that the harmonic polynomial $\im((\eta a_k + \overline{\eta a_{-k}})z^k)$ is nonconstant and shares a nonzero root with the harmonic polynomial $\re(\eta a_n z^n)$.  
\end{enumerate} 
As a partial converse: if (a) holds, then the zero set of $P$ is unbounded. 
\end{theorem}
 
Although part (b) of Theorem~\ref{unbounded-thm} is convoluted, it is not difficult to check in practice because $\eta$ is uniquely determined (up to irrelevant sign) and the zero sets of both harmonic polynomials  involved are simply unions of equally spaced lines   through the origin.
 
\begin{proof} We apply Lemma~\ref{Campbell} to the polynomial $q$ in~\eqref{algebraic-part}, which means letting $b_k = -a_{-k}$ for $k=1, \dots, n$. Since $q$ is not proper, part (b) of the lemma provides two possible scenarios, which are considered below. 

One possibility is that there exists a unimodular constant $\eta$ such that $\eta a_k = - \overline{\eta a_{-k}}$ for $k=1, \dots, n$.  Therefore, for $z\in \mathbb T$ we have  
\[
\re(\eta p(z))  = \re(a_0) +  \sum_{k=1}^n \left( re(\eta a_k z^k + \overline{ \eta a_{-k}} z^k \right) 
= \re(a_0) 
\]
which means that $p(\mathbb T)$ is contained in a line. The converse is true as well. If $p(\mathbb T)$ is contained in a line, then there exists a unimodular constant $\eta$ such that $\re(\eta p)$ is constant on $\mathbb T$. Considering the Fourier coefficients of $\re(\eta p)$, we find $\eta a_k + \overline{\eta a_{-k}}=0$ for all $1\le k\le n$.

The other possibility described in Lemma~\ref{Campbell} (b) transforms into part (b) of Theorem~\ref{unbounded-thm} with the substitution $b_k = -a_{-k}$.
\end{proof}

\section{Intersection of polynomial images of the circle}\label{sec:intersect}

As an application of Theorem~\ref{algebraic-thm}, we establish an upper bound for the number of intersections between two images of the unit circle $\mathbb T$ under Laurent polynomials. It is necessary to exclude some pairs of polynomials from consideration, because, for example, the images of $\mathbb T$ under any two of the Laurent polynomials
\[
p_\alpha (z) = z + z^{-1} + \alpha, \quad -2 < \alpha < 2,
\]
have infinite intersection. This is detected by the computation of polynomial $h$ in Theorem~\ref{algebraic-thm}, according to which $h(x, y) = -4y^2$ regardless of $\alpha$.   

\begin{theorem}\label{intersection-thm}
Consider two Laurent polynomials
\[
p(z) = \sum_{k=-m}^n a_k z^k \quad \text{and} \quad 
\widetilde{p}(z) = \sum_{k=-r}^s b_k z^k
\]
where $m, r\ge 0$, $n, s \ge 1$, and $a_{-m}a_nb_{-r}b_s \ne 0$. Then the intersection $p(\mathbb T)\cap \widetilde{p}(\mathbb T)$ consists of at most $4ns - 2(n-m) (s-r)$ points unless the corresponding polynomials $h$ and $\tilde h$ from Theorem~\ref{algebraic-thm} have a nontrivial common factor.  
\end{theorem}

In the special case of algebraic polynomials, $m=r=0$, the estimate in Theorem~\ref{intersection-thm} simplifies to $2ns$. In this case the theorem is due to Quine~\cite[Theorem 3]{Quine76}, where the bound $2ns$ is shown to be sharp.  A related problem of counting the self-intersections of $p(\mathbb T)$ was addressed in~\cite{Quine73} for algebraic polynomials and in~\cite{KovalevKalmykov} for Laurent polynomials.

\begin{proof}
Let $h_\mathbb C\in \mathbb C[w, \overline{w}]$ be the polynomial   associated to $p$ by Theorem~\ref{algebraic-thm}~(b). 
Consider its homogenization 
\[H(w,\overline{w},\zeta)=\zeta^{2n}h_{\mathbb C}(w/\zeta,\overline{w}/\zeta).\]
Since $h_{\mathbb C}$ has degree $m+n$ in the variable $w$, it follows that $H$ has a zero of order at least $2n-(m+n) = n-m$ at the point $(1, 0, 0)$ of the projective space $\mathbb C\mathbb P^2$. Similarly, it has a zero of order at least $n-m$ at the point $(0, 1, 0)$. 

The homogeneous polynomial $\widetilde{H}$ associated with $\widetilde{p}$ has zeros of order at least $s-r$ at the same two points. Therefore, the projective curves $H=0$ and $\widetilde{H} = 0$ intersect with multiplicity at least $(n-m)(s-r)$ at each of the points $(1, 0, 0)$ and $(0, 1, 0)$ (Theorem~5.10 in \cite[p.~114]{Walker}). 

Bezout's theorem implies that, unless $H$ and $\widetilde{H}$ have a nontrivial common factor, the projective curves $H=0$ and $\widetilde{H}=0$ have at most 
$\deg H \deg \widetilde{H} = 4ns$ intersections in $\mathbb C\mathbb P^2$, counted with multiplicity. Subtracting the intersections at two aforementioned points, we are left with at most  
 $4ns - 2(n-m) (s-r)$ points of intersection in the affine plane. \end{proof}

\bibliographystyle{amsplain} 
\bibliography{references.bib}

\end{document}